\documentclass{amsart}
\usepackage[margin=1in]{geometry}
\usepackage[foot]{amsaddr}
\usepackage{mathabx}
\usepackage{url}
\usepackage{placeins}

\makeatletter
\renewcommand*{\top}{%
  {\mathpalette\@transpose{}}%
}
\newcommand*{\@transpose}[2]{%
  \raisebox{\depth}{$\m@th#1\scriptscriptstyle\mathsf{T}$}%
}
\makeatother

\usepackage{enumitem}
\usepackage{graphicx} 

\usepackage[hypertexnames=false,colorlinks=true,breaklinks=true,bookmarks=true,urlcolor=blue,citecolor=blue,linkcolor=blue,bookmarksopen=false,draft=false]{hyperref}

\newtheorem{theorem}{Theorem}%[section]
\newtheorem{remark}[theorem]{Remark}

\begin{document}

\title[Dual-path fixing and set cover]{The dual-path fixing strategy and its\\ application to the set-covering problem}  

\author{\rm Paulo Michel F. Yamagishi}
\address{P. Yamagishi: Universidade Federal do Rio de Janeiro, ypaulo@gmail.com}

\author{Marcia Fampa}
\address{M. Fampa: Universidade Federal do Rio de Janeiro, fampa@cos.ufrj.br}

\author{Jon Lee}
\address{J. Lee: University of Michigan, jonxlee@umich.edu}

\begin{abstract}
We introduce the \emph{dual-path fixing strategy} to
exploit dual algorithms for solving relaxations of 
mixed-integer nonlinear-optimization problems. 
Such dual algorithms are naturally applied in the
context of branch-and-bound, and eventual impact on
the success of branch-and-bound is our strong motivation. 
Our fixing strategy aims to be more powerful than 
the common strategy of fixing variables
based on 
a single dual-feasible solution (e.g., 
standard reduced-cost fixing for mixed-integer linear optimization),
but to be much faster than ``strong fixing'', essentially requiring
no more time than that of the dual algorithm that we exploit.
We have successfully tested our ideas on mixed-integer linear-optimization set-covering instances 
from the literature, in the context of the dual-simplex method
applied to the continuous relaxations.
\end{abstract}

\maketitle

%%%%%%%%%%%%%%%%%%%%%%%%%%%%%%%%%%%%%%%%%%%%%%%

\section{Introduction}
\label{sec:introduction}

A well-known and powerful technique that has been used to enhance the
performance of mixed-integer linear-optimization (MILO) software is
``reduced-cost fixing'', known already in the early 1980's (see  \cite[Sec.~3.3]{reducedcostfixing} and \cite[Sec. 2]{BalasMartin}). For a minimization problem, the method
seeks to fix integer variables at either their lower or upper bounds,
using the best available objective-value upper bound (i.e., the objective value of
the incumbent feasible solution) and an \emph{optimal} solution of the
dual of the continuous relaxation. In a branch-and-bound (B\&B) context,
which is the work-horse exact algorithm for MILO, 
the information to fix, for each B\&B subproblem, is immediately available
upon solving the continuous relaxation of the subproblem
(as all common linear-optimization solvers provide primal-dual certificates
of optimality). So,
with almost no additional effort, reduced-cost fixing is employed, and
often with great effect (i.e., many variables fixed). 
More generally, for mixed-integer \emph{nonlinear} optimization (MINLO)
problems, in particular some arising in experimental design,
an optimal solution (or near-optimal feasible solution) of 
a convex dual has been exploited to fix variables
(see e.g., \cite{AFLW_Using,Anstreicher_BQP_entropy,Kurt_linx,ChenFampaLee_Fact,PonteFampaLeeMPB}). Again, the motivation is to fix integer variables (typically at 0 or 1)
using information that is already available in the context of B\&B. 

In the general context of 
MINLO, it is natural to consider whether we can easily fix more variables than 
can be fixed using a single optimal solution (or a single near-optimal feasible solution)
of the dual of the continuous relaxation. 
Narrowing the focus to MILO,  it is natural to consider whether we can easily fix more variables than 
reduced-cost fixing is able to fix. Motivated by a modern transportation 
application of set-covering type (MILO instances), \cite{damcov}
developed and tested \emph{strong fixing}, which aims to fix, one by one,
integer variables to their upper or lower bounds, by solving a linear-optimization
problem that is formulated with precisely such a goal. In effect, given an integer variable and an upper or lower bound, strong fixing finds a feasible solution 
to the dual of the continuous relaxation 
that can fix that variable at that bound, \emph{whenever such a 
dual-feasible solution exists.} Typically, strong fixing can fix more variables than reduced-cost fixing, but it is very expensive to carry out.
\cite{damcov} developed methods aimed at speeding up the process, but
practicality was not achieved. Considering again the broader MINLO
context, employing strong fixing can be more prohibitive,
because convex nonlinear-optimization relaxations are typically more expensive to solve than linear-optimization problems. In fact, \cite{damcov} also 
developed strong fixing for a  mixed-integer
second-order-cone formulation, but with less computational success
than they saw for their linear formulation. Another related work is \cite{Cire}, which proposes solving a 0/1 program to obtain a single dual-feasible solution that maximizes the number of fixed variables. This approach is also computationally expensive and does not guarantee the maximum reduction achieved by strong fixing.

Our contribution is to introduce, develop and initiate testing \emph{dual-path fixing}.
We can see it as a method inspired by both the success of reduced-cost fixing
for MILO (and similarly, the success of fixing based on a single good dual solution for MINLO instances that come from the area of experimental design),
and the power of strong fixing. Generally, our goal is to fix more
variables than can be achieved by using a single feasible solution
of the dual of the continuous relaxation with no significant additional cost. 
Our idea is to exploit information 
that is available while one is in the \emph{algorithmic process} of
calculating an optimal solution of a continuous relaxation.
In the context of B\&B, this would typically be a ``dual algorithm''.
Concretely, for MILO, B\&B subproblems are typically solved by the
dual-simplex algorithm, and so what we have available is a sequence of
dual-feasible solutions (to the continuous relaxation). 
As was observed and exploited in \cite{damcov}, it is only 
dual \emph{feasibility} that is needed to fix, and not dual \emph{optimality}
(which is employed by reduced-cost fixing). For some inspiration on the broader applicability of our idea,
we have \cite{Mitchell2,Mitchell1} 
which applied a variant of Karmarkar's algorithm to
the duals of relaxations for a 0/1 linear formulation of a 
combinatorial-optimization problem, in the context of a pure cutting-plane approach. 

We wish to stress that, conceptually, our method applies very generally,
to even \emph{nonconvex} MINLO.
In that context, the Lagrangian dual is employed
(see e.g., \cite{FLbook} for how this works in our context), 
generalizing the
dual of the continuous relaxation of a MILO instance. 
In \S\ref{sec:lagrange}, we demonstrate how
variable fixing based on Lagrangian duality works in general, for 0/1 MINLO.
Further, we indicate how this particularizes for 0/1 MILO, and further particularizes
to reduced-cost fixing. 
In \S\ref{sec:SetCover}, we specialize, enhance and test dual-path fixing for
MILO formulations of set-covering instances.
In \S\ref{sec:outlook}, we describe ideas for further exploration. 

\medskip
\noindent{\bf Notation.}
We denote an all-one vector by $\mathbf{e}$, and 
 a zero-vector simply by $0$. For matrix $X$, we denote column $j$ by $X_{\cdot j}$\,.

%%%%%%%%%%%%%%%%%%%%%%%%%%%%%%%%%%%%%%%%%%%%%%%%%%%%%%%%%%

\section{Fixing variables in binary nonlinear programs}\label{sec:lagrange}

In this section, we develop our ideas for 0/1 MINLO formulations. 
The extension to arbitrary integer variables
with bounds is straightforward, but we concentrate on 
the 0/1 situation for pedagogical simplicity and
because our target applications all have 0/1 integer variables.
As for testing, which we report on in \S\ref{sec:SetCover},
we will confine our experiments to challenging 0/1 MILO formulations
of set-covering instances from the literature. We leave it as 
future work, to broad the scope of our experiments, based on what we present in
the present section.

\subsection{Lagrangian duality and fixing for 0/1 MINLO}\label{sec:LagrangeFixing}

Consider the fairly-general  binary nonlinear-optimization problem
\begin{equation}
	\label{BNLP}\tag{\mbox{BNLP}}
	\begin{array}{llll}
		 \min  & f(x,y)\\
		\mbox{s.\ t.} \; & g(x,y)  \leq  0,  \\
		               & h(x,y)=0, \\
        &x  \in \{0,1\}^{n},~y  \in \mathbb{R}^{q},
	\end{array}
\end{equation}
where  $f: \mathbb{R}^{n}\times \mathbb{R}^{q}\rightarrow \mathbb{R}$, $g: \mathbb{R}^{n}\times \mathbb{R}^{q}\rightarrow \mathbb{R}^m$, 
$h: \mathbb{R}^{n}\times \mathbb{R}^{q}\rightarrow \mathbb{R}^p$,
its continuous relaxation
\begin{equation}
	\label{relP}\tag{$\widebar{\mbox{P}}_{\mbox{\tiny NLP}}$}
	\begin{array}{llll}
		\min  & f(x,y)\\
	\mbox{s.\ t.} \; & g(x,y)  \leq  0,  \\
		               &  h(x,y)=0, \\
        &0\leq x\leq \mathbf{e},
	\end{array}
\end{equation}
its associated Lagrangian function $\mathcal{L}:\mathbb{R}^{n}\times \mathbb{R}^{q}\times \mathbb{R}^m\times \mathbb{R}^p\times \mathbb{R}^n\times \mathbb{R}^n\rightarrow \mathbb{R}$,
\[
\mathcal{L}(x,y,\lambda,\mu,\upsilon,\nu)=f(x,y)+  \lambda^\top g(x,y) +   \mu^\top  h(x,y) - \upsilon^\top x + \nu^\top (x-\mathbf{e}),
\]
and, with $\mathcal{D}:= \mbox{dom}\, f \cap  \mbox{dom}\, g \cap  \mbox{dom}\, h$, its associated Lagrangian dual problem 
\begin{equation}
\label{dualP}\tag{$\widebar{\mbox{D}}_{\mbox{\tiny NLP}}$}
	\begin{array}{llll}
		& \displaystyle\max_{\lambda,\mu,\upsilon,\nu}  & -\nu^\top \mathbf{e} + \displaystyle \inf_{x\in\mathcal{D}} f(x,y)+  \lambda^\top g(x,y) +   \mu^\top  h(x,y) - (\upsilon - \nu)^\top x\\
		& \mbox{s.\ t.} \; & \lambda\geq 0,~\upsilon\geq 0,~\nu  \geq  0.
	\end{array}
\end{equation}
\begin{theorem}\label{thm:fixBNLP}
    Let
\begin{itemize}
    \item {\rm UB} be the objective-function value of a feasible solution for {\rm\ref{BNLP}},
    \item $(\hat{\lambda},\hat{\mu},\hat{\upsilon},\hat{\nu})$ be a feasible solution for  {\rm\ref{dualP}} with objective-function \hbox{value $\hat{\zeta}$}. 
\end{itemize}
Then, for every optimal solution  $(x^*,y^*)$ for {\rm\ref{BNLP}}, we have:
\[
\begin{array}{ll}
x_k^*=0, ~ \forall  k\in N \mbox{ such that } \mbox{\rm UB}-\hat{\zeta} < \hat{\upsilon}_k\,,\\
x_k^*=1, ~ \forall  k\in N \mbox{ such that } \mbox{\rm UB}-\hat{\zeta} < \hat{\nu}_k\,.\\
\end{array}
\]
\end{theorem}

\begin{proof}
Consider \ref{relP} with the constraint $x_k\geq 0$ replaced by $x_k\geq 1$. The Lagrangian dual problem becomes then 
\begin{equation}
	\label{dualPmod}
	\begin{array}{llll}
		& \displaystyle\max_{\lambda,\mu,\upsilon,\nu}  & \upsilon_k-\nu^\top \mathbf{e} + \displaystyle \inf_{x\in\mathcal{D}} f(x,y)+  \lambda^\top g(x,y) +   \mu^\top  h(x,y) - (\upsilon - \nu)^\top x\\
		& \mbox{s.\ t.} \; & \lambda\geq 0,~\upsilon\geq 0,~\nu  \geq  0,
	\end{array}
\end{equation}
Similarly, consider \ref{relP} with the constraint $x_k\leq 1$ replaced by $x_k\leq 0$. The Lagrangian dual problem becomes then 
\begin{equation}
	\label{dualPmoda}
	\begin{array}{llll}
		& \displaystyle\max_{\lambda,\mu,\upsilon,\nu}  & \nu_k-\nu^\top \mathbf{e} + \displaystyle \inf_{x\in\mathcal{D}} f(x,y)+  \lambda^\top g(x,y) +   \mu^\top  h(x,y) - (\upsilon - \nu)^\top x\\
		& \mbox{s.\ t.} \; & \lambda\geq 0,~\upsilon\geq 0,~\nu\geq  0,
	\end{array}
\end{equation}
Notice that if $(\hat{\lambda},\hat{\mu},\hat{\upsilon},\hat{\nu})$ is a  feasible solution for \ref{dualP}\,, then it is a feasible solution of the modified dual problems \eqref{dualPmod} and \eqref{dualPmoda} as well.

The objective value of \eqref{dualPmod} at this solution is
 $\hat{\zeta}+\hat{\upsilon}_k$\,. 
Therefore $\hat{\zeta}+\hat{\upsilon}_k$
 is a lower bound on the objective value of solutions of \ref{BNLP} having $x_k=1$. So if $\hat{\zeta}+\hat{\upsilon}_k>UB$, then
 no optimal solution of \ref{BNLP}  can have $x_k=1$.  
Analogously, the objective value of \eqref{dualPmoda} at this solution is
 $\hat{\zeta}+\hat{\nu}_k$\,.  Therefore $\hat{\zeta}+\hat{\nu}_k$
 is a lower bound on the objective value of solutions of \ref{BNLP} having $x_k=0$. So, if $\hat{\zeta}+\hat{\nu}_k>{\rm UB}$, then
 no optimal solution of \ref{BNLP}  can have $x_k=0$. 
\end{proof}

%%%%%%%%%%%%%%%%%%%%%%%%%%%%%%%%%%%%%%%%%%%%%%%

\subsection{Strong fixing for 0/1 MINLO}

Theorem \ref{thm:fixBNLP} gives us a methodology to fix variables at 0 and 1  in \ref{BNLP}, considering the objective-function value of a feasible solution and \emph{any} feasible solution for the Lagrangian dual problem \ref{dualP} of its continuous relaxation. 
It is easy to see from Theorem \ref{thm:fixBNLP} that 
a better (i.e., lower) UB gives us the potential to fix more variables. 

In the \emph{strong-fixing} (SF) methodology
(see \cite{damcov}), for all $j\in N$, we solve the problem where we maximize the objective function of \ref{dualP} plus $\upsilon_j$ (resp., $\nu_k$), subject to $(\lambda,\mu,\upsilon,\nu)$ being a feasible solution to \ref{dualP}\,.
For each $j\in N$, if there is a feasible solution $(\hat\lambda,\hat\mu,\hat\upsilon,\hat\nu)$ to \ref{dualP} that can be used in Theorem \ref{thm:fixBNLP} to fix $x_j$ at 0 (resp., at 1), then  the optimal solution of this problem has objective value  greater than UB and can be used as well. 

We can see SF as following the philosophy expounded as far back as
the early 1980's in \cite[Sec.~3.3]{reducedcostfixing}:
``The fact that we are willing to abandon a
 search tree and restart the entire procedure with a reduced problem
 reflects our philosophy that we would rather do almost anything than
 perform branch and bound. Our philosophy is not a prejudice; it is based
 on our experience solving this type of problem.'' Still, there have to
 be some limits on the work involved in reducing, and the next section
 contains our proposal to get some of the power of SF, but with 
 work closer to that of RCF.

%%%%%%%%%%%%%%%%%%%%%%%%%%%%%%%%%%%%%%%%%%%%%%%

\subsection{Dual-path fixing for 0/1 MINLO}

We typically see the use of the optimal solution of \ref{dualP} in Theorem \ref{thm:fixBNLP}, especially when applied to a binary linear-optimization problem. Nevertheless, any feasible dual solution can be used. We therefore propose that a dual-feasible method be employed to solve the continuous relaxation of \ref{BNLP}, which generates a feasible dual solution in each iteration. Finally, we propose using all dual solutions found during the iterations of the method in Theorem \ref{thm:fixBNLP}, with the aim of creating new opportunities to fix variables without significant additional computational effort. We refer to this technique as \emph{dual-path fixing} (DPF).

%%%%%%%%%%%%%%%%%%%%%%%%%%%%%%%%%%%%%%%%%%%%%%%

\subsection{Particularization to pure-0/1 MILO}

Next, we particularize the methodologies above for the binary linear-optimization problem:
\begin{equation}
	\label{BLP}\tag{\mbox{BLP}}
		 \min  \{ c^\top x
	~:~ Ax  \geq  b,\quad Cx=d,\quad
        x  \in \{0,1\}^{n}\},
\end{equation}
where  $A\in\mathbb{R}^{m\times n}$ and $C\in\mathbb{R}^{p\times n}$. The
continuous relaxation of \ref{BLP} and its dual problem are
\begin{align}		 
&~\displaystyle\min_{x}  \{ c^\top x
	~:~ Ax  \geq  b,\quad Cx=d,\quad
        0\leq x\leq \mathbf{e}\},\label{relP-lp}\tag{$\widebar{\mbox{P}}_{\mbox{\tiny LP}}$}\\
&\displaystyle\max_{\lambda,\mu,\upsilon,\nu} \{  \lambda^\top b + \mu^\top d - \nu^\top \mathbf{e}~:~
		  \lambda^\top A + \mu^\top C + \upsilon^\top  - \nu^\top  = c^\top, \quad \lambda\geq 0, \quad \upsilon\geq 0,\quad \nu  \geq  0\}. \label{dualP-lp}\tag{$\widebar{\mbox{D}}_{\mbox{\tiny LP}}$}
\end{align}

Let $\mbox{UB}$ be the objective-function value of a feasible solution for \ref{BLP}, and  $(\hat{\lambda},\hat{\mu},\hat{\upsilon},\hat{\nu})$ be a feasible solution for   \ref{dualP-lp} with objective-function \hbox{value $\hat{\zeta}$}. 
Then, for every optimal solution $x^*$ for \ref{BLP}, and from Theorem \ref{thm:fixBNLP} we have:
\[
\begin{array}{ll}
x_k^*=0, ~ \forall  k\in N \mbox{ such that }  \mbox{UB}-\hat{\zeta} <  c_k - (\hat{\lambda}^\top A_{\cdot k} + \hat{\mu}^\top C_{\cdot k} -\hat{\nu}_k),\\
x_k^*=1, ~ \forall  k\in N \mbox{ such that }  \mbox{UB}-\hat{\zeta} <  \hat{\lambda}^\top A_{\cdot k} + \hat{\mu}^\top C_{\cdot k} +\hat{\upsilon}_k - c_k~,
\end{array}
\]
where we use the equality constraint in \ref{dualP-lp} to extract the values of $\hat{\upsilon}_k$ and $\hat{\nu}_k$\,. Equivalently, we have:
\begin{align}
&\!\!\!\!\!\!x_k^*\!\leq \!\displaystyle \left\lfloor \frac{\mbox{UB}-\hat{\zeta}}{c_k -(\hat{\lambda}^\top A_{\cdot k} + \hat{\mu}^\top C_{\cdot k} -\hat{\nu}_k)}\right\rfloor, ~ \forall  k\in N \mbox{ such that }   c_k -  (\hat{\lambda}^\top A_{\cdot k} + \hat{\mu}^\top C_{\cdot k} -\hat{\nu}_k)\! > \!0,\label{fixlp0}\\
&\!\!\!\!\!\!x_k^*\!\geq \!1 \!-\! \displaystyle \left\lfloor \frac{\mbox{UB}-\hat{\zeta}}{\hat{\lambda}^\top A_{\cdot k} + \hat{\mu}^\top C_{\cdot k} +\hat{\upsilon}_k -c_k}\right\rfloor, ~ \forall  k\in N \mbox{ such that }     c_k -  (\hat{\lambda}^\top A_{\cdot k} + \hat{\mu}^\top C_{\cdot k} +\hat{\upsilon}_k)\!<\!0.\label{fixlp1}
\end{align}

\begin{remark}
If $(\hat{\lambda},\hat{\mu},\hat{\upsilon},\hat{\nu})$ is an optimal solution for \ref{dualP-lp}\,, from complementary slackness, we cannot simultaneously have  $\hat\upsilon_k>0$ and $\hat{\nu}_k>0$. Therefore, in this case, \eqref{fixlp0} and \eqref{fixlp1}  reduce, respectively,  to: 
\begin{align}
&x_k^*\leq \displaystyle \lfloor
(\mbox{UB}-\hat{\zeta})/\bar{c}_k)
\rfloor, 
~ \forall  k\in N \mbox{ such that }   \bar{c}_k > 0,\label{ug0}\\%[10pt]
&x_k^*\geq 1 - \displaystyle \lfloor 
-(\mbox{UB}-\hat{\zeta})/\bar{c}_k)
\rfloor, ~ \forall  k\in N \mbox{ such that }     \bar{c}_k < 0,\label{vg0}
\end{align}
where $\bar{c}_k:= c_k -  (\hat{\lambda}^\top A_{\cdot k} + \hat{\mu}^\top C_{\cdot k})$ is the reduced cost associated to the variable $x_k$ in the continuous relaxation of \ref{BLP}, for all $k\in N$. Note that, in this case, \eqref{ug0} corresponds to the case where $\hat{\upsilon}_k=\bar{c}_k>0$ and $\hat{\nu}_k=0$, and \eqref{vg0} corresponds to the case where $\hat{\nu}_k=-\bar{c}_k>0$ and $\hat{\upsilon}_k=0$. The methodology for fixing variables based in Theorem \ref{thm:fixBNLP}, using an optimal solution of \ref{dualP-lp}\,, is well known as reduced-cost fixing (RCF).

It is common for the constraint $x\leq \mathbf{e}$ to be redundant in \ref{relP-lp} and to be omitted. In this case, there is no dual variable $\nu$ in \ref{dualP-lp}, and \eqref{fixlp0} reduces to \eqref{ug0} for any feasible dual solution.
\end{remark}

%%%%%%%%%%%%%%%%%%%%%%%%%%%%%%%%%%%%%%%%%%%%%%%

\subsubsection{Strong fixing for pure-0/1 MILO}\label{subsec:SFMILO}

When considering the particular case of the binary linear-optimization problem \ref{BLP}, the SF methodology requires the solution of the following problems for each $j\in N$. 
\begin{align}
&	\begin{array}{llll}
		& \displaystyle\max_{\lambda,\mu,\upsilon,\nu}  &  c_j - ({\lambda}^\top A_{\cdot j} + {\mu}^\top C_{\cdot j} -{\nu}_j) +  \lambda^\top b + \mu^\top d - \nu^\top \mathbf{e}  \\
		& \mbox{s.\ t.} \; &   \lambda^\top A + \mu^\top C + \upsilon^\top  - \nu^\top  = c^\top, \\
        &&\lambda\geq 0,~\upsilon\geq 0,~\nu  \geq  0.
	\end{array}\label{fjlp0}\tag{F$_j^{\mbox{\tiny BLP(0)}}$}\\
&	\begin{array}{llll}
		& \displaystyle\max_{\lambda,\mu,\upsilon,\nu}  &  {\lambda}^\top A_{\cdot j} + {\mu}^\top C_{\cdot j} +{\upsilon}_j - c_j +  \lambda^\top b + \mu^\top d - \nu^\top \mathbf{e}  \\
		& \mbox{s.\ t.} \; &   \lambda^\top A + \mu^\top C + \upsilon^\top  - \nu^\top  = c^\top, \\
        &&\lambda\geq 0,~\upsilon\geq 0,~\nu  \geq  0.
	\end{array}\label{fjlp1}\tag{F$_j^{\mbox{\tiny BLP(1)}}$}
\end{align}
For a given $j\in N$, we should have that the right-hand side in \eqref{fixlp0} equal to 0 to be able to fix the variable $x_j$ at 0 in \ref{BLP}. Equivalently, we should have the  optimal value of \ref{fjlp0} greater than UB  to be able to fix the variable $x_j$ at 0 in \ref{BLP} using the optimal solution of \ref{fjlp0} in Theorem \ref{thm:fixBNLP}.

Analogously, for a given $j\in N$, we should have that the right-hand side in \eqref{fixlp1} equal to 1 to be able to fix the variable $x_j$ at 1 in \ref{BLP}. Equivalently, we should have the  optimal value of \ref{fjlp1} greater than UB  to be able to fix the variable $x_j$ at 1 in \ref{BLP} using the optimal solution of \ref{fjlp1} in Theorem \ref{thm:fixBNLP}.

%%%%%%%%%%%%%%%%%%%%%%%%%%%%%%%%%%%%%%%%%%%%%%%

\subsubsection{Dual-path fixing for 0/1 MILO}

DPF particularizes to MILO depending on 
the algorithm used to solve continuous relaxations.
A very natural choice for us is the dual-simplex method,
which is commonly and naturally applied by widely-used codes within both cutting-plane and B\&B algorithms. But additionally, for example, interior-point methods
can be applied to the duals of continuous relaxations, as was done in a cutting-plane approach for graph matching formulations; see \cite{Mitchell2,Mitchell1}.

%%%%%%%%%%%%%%%%%%%%%%%%%%%%%%%%%%%%%%%%%%%%%%%

\section{Set-covering MILO}\label{sec:SetCover}

We decided to make a detailed study of our ideas 
for a class of difficult and well-motivated set of \ref{BLP}
instances.
In the set-covering problem, we consider a set of elements and a collection of subsets of these elements, each subset having an associated cost. The objective is to select subsets whose total cost is to be minimized and that together cover (i.e., contain) all elements of the original set (see, e.g., \cite{Balas1980b}). The set-covering problem is a classic combinatorial-optimization problem, fundamental to computational-complexity theory and with applications in a wide variety of areas. The decision version belongs to the class of NP-complete problems, as 
is well known and was demonstrated in Karp's monumental paper \cite{Karp72}.

The set-covering problem can be formally defined as follows. Let $w$ be a positive real $n$-dimensional vector and $A$ be an $m \times n$ 0/1-valued matrix. Let $M := {1,2,\ldots,m}$ and $N := {1,2,\ldots,n}$. The value $w_j$ represents the cost of column $j$ of $A$, for $j \in N$. We say that column $j \in N$ covers row $i \in M$ of $A$ if the entry $(i,j)$ of $A$ is equal to 1. The set-covering problem seeks a subset $S \subseteq N$ of columns of minimum cost such that each row $i \in M$ is covered by at least one column $j \in S$. It is naturally formulated as the following 0/1 linear-optimization problem:
\begin{equation}\label{scp}\tag{SCP}
\min\{w^\top z : Az \geq \mathbf{e},\ z \in \{0,1\}^n\}.
\end{equation}
For the variable, we use $z$ rather than $x$ to be consistent with 
\cite{damcov}.

With the objective of finding the optimal solution of \ref{scp}, the B\&B algorithm, first introduced by \cite{Land1960} for mixed-integer linear optimization, is one of the most used and fundamental exact methods; see \cite[Sec.~II.4]{NWbook} for a modern presentation.
The success of the algorithm depends on the computation of good lower and upper bounds for {\rm\ref{scp}}; the former can be obtained by solving the continuous relaxation of \ref{scp}, and the latter is given by a feasible solution to the problem.

In the case of \ref{scp},  its continuous relaxation \ref{relSCP} and respective dual problem \ref{d} can be formulated as:
\begin{align}
&\min\{w^\top z : Az \geq \mathbf{e},\ z \geq 0\}, \label{relSCP}\tag{P}\\
&\max\{u^\top \mathbf{e} : u^\top A \leq w^\top,\ u \geq 0\}.\label{d}\tag{D}
\end{align}
As is common for set-covering formulations, 
we relax $z\in \{0,1\}^n$ to
$z\geq 0$, because the ``objective-function pressure''
related to the positivity of $w$,
together with the $Az\geq \mathbf{e}$ constraints,
and the nonnegativity of $A$ 
ensure that no variable would ever be greater that
1 in an optimal solution of either \ref{scp} or
\ref{relSCP}. 

%%%%%%%%%%%%%%%%%%%%%%%%%%%%%%

\subsection{Variable fixing for set covering}

\noindent {\bf Reduced-cost fixing for set covering.}
The RCF technique is based on Theorem \ref{thm:fix_dopt}, which is derived from \eqref{ug0}.

\begin{theorem}\label{thm:fix_dopt}
Let {\rm UB} be the objective function value for a feasible solution to {\rm\ref{scp}}, and let $u^\star$ be 
%a feasible 
an optimal solution to {\rm\ref{d}}. For all $j\in N$, consider the reduced cost associated to variable $x_j$ of {\rm\ref{relSCP}}, defined as $\bar{w}_j:=w_j- {u^\star}^\top A_{\cdot j}$\,. For every optimal solution $z^\star$ to {\rm\ref{scp}}, we have:
\begin{align}
&z_j^\star \leq \left\lfloor
({\rm UB}-{u^\star}^\top\mathbf{e})/\bar{w}_j 
\right\rfloor ,\quad \forall j \in N \text{ such that } \bar{w}_j
> 0.\label{desigualdades_de_fixacao}
\end{align}
\end{theorem}
For $j \in N$, the right-hand side of \eqref{desigualdades_de_fixacao} (a nonnegative integer) must be equal to 0 in order to fix $z_j$ to 0
in \ref{scp}.

We note that Theorem \ref{thm:fix_dopt}, for reduced-cost fixing,  considers the specific use of the optimal solution to \ref{d}.  However, as we omit $z\leq\mathbf{e}$ in \ref{relSCP}, the result of the theorem applies when considering any  feasible solution to \ref{d} (see last comment before  \S\ref{subsec:SFMILO}). 

\medskip
\noindent {\bf Dual-path fixing for set covering.}
For DPF applied to \ref{scp}, we employ the dual-simplex algorithm to \ref{relSCP} (equivalently, the primal-simplex algorithm on \ref{d}), and we use the entire sequence of dual feasible solutions generated by the algorithm to potentially fix variables via Theorem \ref{thm:fix_dopt}. 

\medskip
\noindent {\bf Strong fixing  for set covering.}
The SF technique for \ref{scp} was considered in \cite{damcov} and leads to the greatest number of fixed variables when considering any dual-feasible solution in Theorem \ref{thm:fix_dopt}.
For a given $j\in N$, we should have that the right-hand side in \eqref{desigualdades_de_fixacao} equal to 0 to be able to fix the variable $z_j$ at 0 in \ref{scp}. Equivalently, we should have
$
 w_j+ \hat{u}^{\star\top}(\mathbf{e} -  A_{\cdot j})> UB\,.
$
Once more, we observe that any feasible solution $\hat{u}$ for \ref{d} can be used in \eqref{desigualdades_de_fixacao}. Then, for all $j\in N$, the SF technique employs the solution of 
\begin{align}
    \mathfrak{z}_{j}^{\mbox{\tiny SCP(0)}}:=w_j+\max\{u^\top (\mathbf{e} -  A_{\cdot j})~:~ u^\top A\leq  w^\top,~ 
u\geq 0\}. \label{fj}\tag{F$_j^{\mbox{\tiny SCP(0)}}$}
\end{align}

Note that, for each $j\in N$, if there is a feasible solution $\hat{u}$ to \ref{d} that can be used in \eqref{desigualdades_de_fixacao} to fix $z_j$ at 0, then  the optimal solution of \ref{fj} has objective value $\mathfrak{z}_{j}^{\mbox{\tiny SCP(0)}}$ 
greater than UB and can be used as well. 
In the SF technique, for $j=1,2,\ldots,n$, we iteratively solve equation \ref{fj} if variable $j$ has not yet been fixed (using the solutions of \ref{fj} from  previous iterations) and, considering its solution in Theorem \ref{thm:fix_dopt}, we fix all variables that are possible 
to be fixed.

%%%%%%%%%%%%%%%%%%%%%%%%%%%%%%%%%%%%%%%%%%%%%%%

\subsection{Dominated row elimination (DRE) for set covering}

For general MILO instances, preprocessors seek to remove 
redundant constraints (see, for example, \cite[Sec.~1.1]{Savelsbergh}). And after variable fixing, again 
redundant constraints may arise. For set covering, 
a particularly simple kind of redundancy commonly arises, 
and there is a standard and simple way of handling it
which we now describe.

We say that a row $\hat \imath$ \emph{dominates} another row $i$ if every set that covers $\hat\imath$ also covers $i$, i.e., if $A_{\hat\imath j}\le A_{ij}$ for all $j\in N$.
If row $i$ is dominated by row $\hat\imath $, then any solution that covers $\hat\imath $ will automatically cover $i$.  
 Thus, row $i$ can be removed from the problem.

So, for \ref{scp}, we consider the  following dominated-row elimination procedure  (DRE), that we iteratively  apply to a given set-covering constraint matrix $A$ until no row is eliminated.
\begin{enumerate}[label=\Roman*]
    \item Eliminate dominated rows (see, for example, \cite[Sec.~3.2(1)]{BEASLEY198785}). 
    \item If a row is covered by only one column, fix the variable corresponding to that column at 1 and eliminate all rows that were covered by the column (see, e.g., \cite[Sec.~3.1(4)]{BEASLEY198785}).
\end{enumerate}

%%%%%%%%%%%%%%%%%%%%%%%%%%%

\subsection{Numerical experiments for set covering}\label{sec:numexp}

To investigate the effectiveness of the  dual-path fixing (DPF) strategy proposed in this work, when applied to instances of \ref{scp}, we compare it with the well-known reduced-cost fixing (RCF) strategy and with the strong fixing (SF) strategy. Although SF is impractical due to its high computational cost, it provides an ideal benchmark 
by yielding the maximum number of fixed variables by applying the result of Theorem \ref{thm:fixBNLP}. To further evaluate the impact of variable fixing on the emergence of dominated rows, we apply the dominated row elimination (DRE) procedure described in the previous section after each variable-fixing strategy.

Specifically, we compare the procedures described in the following, all aimed at  reducing the size of the instances of \ref{scp}. We assume that the input matrices $A$ of the instances considered initially contain no dominated rows.

%%%%%%%%%%%%%%%%%%%%%%%%%

 \subsubsection{Procedures evaluated}\label{sec:procedures}
\begin{itemize}
\item RCF+DRE: we apply RCF and  then, if we can fix variables, we apply DRE. 
\item DPF+DRE: we apply DPF and then, if we can fix variables, we apply DRE (we apply DRE only once, after fixing all possible variables, using all dual feasible solutions obtained). 
\item $\mathcal{I}$(RCF+DRE): we iteratively apply RCF+DRE until no variable can be fixed.
\item $\mathcal{I}$(DPF+DRE): we iteratively apply DPF+DRE until no variable can be fixed.
\item SF+DRE: we apply SF and then, if we can fix variables, we apply DRE   (we apply DRE only once, after fixing all possible variables using the solutions of all problems \ref{fj} solved. Note that all \ref{fj} have exactly the same set of variables and constraints.
\end{itemize}
For all procedures described above, except SF+DRE, we have directly solved \ref{d} with the primal-simplex method, initializing the algorithm with the slack variables in the basis, i.e., with the solution $u:=0$.

State-of-the-art MILO solvers (e.g., Gurobi) tend
to compartmentalize the linear-opt\-i\-mization (usually simplex-algorithm based) subsolver within the overall MILO 
B\&B (or more generally, branch-and cut) algorithm. 
Because of this, it is not easy for the 
MILO B\&B algorithm to gain access to all dual solutions 
from the iterates of simplex-algorithm
subsolver (which is what we want 
for dual-path fixing). Because of this, we decided to do some 
experiments aimed at investigating the \emph{potential}
for our ideas, were they to be implemented in an 
efficient manner in a state-of-the-art MILO solver. 
For our experiments, we used the well-known open-source 
GLPK (GNU Linear Programming Kit) version 5.0, an ANSI C  callable library 
for linear optimization (and more), distributed
under the GNU GENERAL PUBLIC LICENSE, Version 3, 29 June 2007; see \cite{glpk}. But our ambitious ultimate goal is to provoke the (commercial and open source) solver developers 
to more tightly integrate their (convex) subsolvers (LP, QP, NLP, SOCP, SDP) with their B\&B and branch-and-cut solvers, and then implement our ideas.

We ran the  experiments on a 14-inch MacBook Pro equipped with an Apple M3 Pro chip (11-core CPU), 18 GB of unified memory, and running macOS Tahoe 26.2.

We found it to to be convenient to
apply the primal-simplex algorithm to \ref{d}
(rather than the dual-simplex algorithm to \ref{relSCP}),
and to always start with an all-slack basis, 
which is always feasible because meaningful instances of 
\ref{scp} have positive $w$.
Starting with a consistently pre-defined initial
basis is very convenient for reproducability
of our experiments, and ensures more consistent comparisons across procedures. 

%%%%%%%%%%%%%%%%%%%%%%%%%%%%%%%%%%%%%%%%

\subsubsection{Test instances}

We tested our procedures on 35 instances from the standard non-unicost set-covering test problems in steps 4--6 and A--D from \cite{Balas1980b,BEASLEY198785}, available in the  OR-Library \cite{Beasley1990ORLibrarySCP}. From the results reported below, we excluded 10 instances from sets 4 and 5 (out of 20 total)  because all variables in these instances were already fixed by RCF, leaving no room for improvement with the DPF strategy. We note that the cost vector $w$ is nonnegative for all instances.

We also generated 25 instances as described in \cite[Sec.~5.1]{damcov}. These are instances of a set-covering problem that models a problem of determining  an optimal location for safety-landing-sites (SLSs) for commercial electric Vertical Take-Off and Landing (eVTOLs) vehicles transporting passengers in large cities. The objective is to install the set of SLSs at the lowest possible cost, such that all points in a given air-transportation network are covered by at least one of them. The input data needed to generate each instance are the number $n$ of candidate locations for the SLSs, the number of nodes $\nu$ on the graph that models the air-transportation network, and the interval $[R_{\min},R_{\max}]$ in which the randomly generated covering radii for the SLSs must be.  These covering radii determine the maximum allowed distance between the SLSs and points in the graph that are covered by them. A detailed description of how to construct the input constraint matrices $A$ for the \ref{scp} instances can be found in Algorithm 1 of \cite[Sec.~5.1]{damcov}. For our instances, we define $R_{\min}=0.11$, $R_{\max}=0.19$, and, following \cite{damcov}, we define $\nu=0.3 n$, with $n=500, ~1000, ~1500$.

To fully investigate the effectiveness of the fixing strategies, the upper bounds UB used for fixing variables is always set to the best objective value reported for the Beasley instances and to the optimal value for the SLS instances.

%%%%%%%%%%%%%%%%%%%%%%%%%%%%

\subsubsection{Results}\label{sec:numres}  

In Figures \ref{fig:reduction_Beasleya}--\ref{fig:reduction_SLS}, we show the average percentage reduction in the number of variables for the Beasley and SLS instances. These averages are computed over each set of instances of the same size and for all procedures described in \S\ref{sec:procedures}. 

We observe that RCF is already highly effective across all instance sets, yet DPF further increases the number of variables fixed compared with RCF in every case. We note that the average improvement of 1.4\% for SLS instances with $n = 1500$ corresponds to 21 additional variables fixed by DPF at no significant extra cost. 

Figure \ref{fig:reduction_Beasleya} shows that for Beasley’s instances in sets 4 and 5 there is a consistent increase in variables fixed across the procedures from left to right, and the effectiveness of $\mathcal{I}$(DPF) is supported by its results being close to SF. For set 6,  at the plotted precision, iterations of the iterative procedures produce  no measurable improvement on average; importantly, in this case, a single application of DPF (one simplex run) fixes more variables than $\mathcal{I}$(RCF).

Figure \ref{fig:reduction_Beasleyb} shows that, for Beasley’s larger instances in sets A--D, RCF is even more effective than in the previous sets, leaving less room for improvement by DPF and the iterative procedures. Nevertheless, DPF is still able to increase the number of fixed variables on average for all sets, bringing this number closer to that achieved by SF. Regarding the iterative procedures,  improvements beyond the first iteration are observed only for set A; for the remaining sets, a single application of DPF is more effective than the more computationally expensive iterative application of RCF.

Figure \ref{fig:reduction_SLS} shows  that RCF is less effective for the SLS instances than for the Beasley instances. The number of fixed variable increases for every method from left to the right. Moreover, the iterative algorithms have a substantial impact on the smaller instances, where they led to a notable increase in the number of fixed variables.

Tables \ref{tab:beasley-preprocessing-comparison} and \ref{tab:sls-preprocessing-comparison}, in the Appendix, report for each instance in the Beasley and SLS test sets, respectively, the number of variables $n$ and constraints $m$ obtained after applying  each procedure described in \S\ref{sec:numres}. For the iterative procedures $\mathcal{I}$(RCF+DRE) and $\mathcal{I}$(DPF+DRE), the number of executed  iterations is also reported in brackets. The average results shown in Figures \ref{fig:reduction_Beasleya}--\ref{fig:reduction_SLS} are derived from the data in these tables. 

Figure \ref{fig:simplex_iterations} presents results for an SLS instance originally with $n=1000$ and $m=7585$,  for which both RCF and DPF are highly effective. For this instance, RCF already fixes 831 out of 1000 variables, yet DPF is able to fix an additional 14 variables.  Figure \ref{fig:simplex_iterations} shows the cumulative number of variables fixed by DPF during the execution of the simplex method to solve the dual problem \ref{d}, where 313 iterations were executed, generating 313 candidate dual-feasible solutions to fix variables according to Theorem \ref{thm:fixBNLP}. We note that the first variable is fixed early in the simplex execution, specifically at iteration 32. By iteration 155, 50\% of the total 845 variables fixed by DPF have already been determined. Furthermore, at iteration 241, DPF fixes the first of the 14 variables that cannot be fixed by RCF.  

An important advantage of DPF becomes clear when it is used in combination  with strong-branching in B\&B. Even if the simplex method is interrupted before full convergence — as is done when performing strong branching — the cumulative fixing plot provides a lower bound on the number of variables that will ultimately be fixed for a given branching decision. For example, if the plot in Figure \ref{fig:simplex_iterations} corresponds to a specific subproblem generated by a particular branching strategy, interrupting the simplex method after 200 iterations already guarantees that at least 600 variables will be fixed by DPF for that branching decision, providing critical early insight for the B\&B process. 

Finally, we note that the dual gap at iteration 200 is already very small, so the substantial number of variables fixed from this iteration is due  to the diversity of dual solutions considered, rather than simply a large decrease in the dual gap.

Figures \ref{fig:B_46}--\ref{fig:SLS} show average results for each set of test instances, concerning the  application of  DPF. Similarly to Figure \ref{fig:simplex_iterations}, they display, on the left vertical axis, the average percentage of cumulative fixed variables and, on the right vertical axis,  the average percentage of the initial gap. The horizontal axis represent the percentage of the total number of executed iterations. The conclusions drawn from these figures are analogous to those derived from Figure \ref{fig:simplex_iterations}. We observe that variables begin to be fixed in the early iterations of the simplex method, and also that even after the gaps have been substantially reduced, a large number of variables continue to be fixed using the different dual-feasible solutions obtained throughout the simplex iterations.

Figure \ref{fig:numSimplexit} reports the number of simplex iterations performed at each iteration of $\mathcal{I}$(DPF) for the same instance considered in Figures \ref{fig:simplex_iterations} and  \ref{fig:iterative_methods}. By comparing Figures \ref{fig:iterative_methods} and \ref{fig:numSimplexit}, we observe that as the iterations of $\mathcal{I}$(DPF) progress, the size of the problem solved decreases and, consequently, the number of simplex iterations is reduced. Nevertheless, Figure \ref{fig:iterative_methods} shows that DPF fixes more variables than RCF at every iteration of their respective iterative procedures until $\mathcal{I}$(DPF) converges. 

Figure \ref{fig:iterative_methods} compares the number of variables fixed at each iteration by the iterative methods $\mathcal{I}$(RCF+DRE) and $\mathcal{I}$(DPF+DRE). The instance  considered is the same as in Figure \ref{fig:simplex_iterations}. The comparison highlights several advantages of $\mathcal{I}$(DPF+DRE) over $\mathcal{I}$(RCF+DRE) when applied to this instance. From the  first iteration, $\mathcal{I}$(DPF+DRE) fixes more variables than $\mathcal{I}$(RCF+DRE), and this difference grows as the iterations progress. Ultimately, $\mathcal{I}$(DPF+DRE) requires fewer iterations to converge while fixing a significantly larger number of variables, even when compared with $\mathcal{I}$(RCF+DRE) after completing more iterations.  The horizontal line at 14 represents the number of variables fixed by SF and serves as a benchmark for how closely each  each iterative procedure approaches the ideal number of fixed variables.

%%%%%%%%%%%%%%%%%%%%%%%%%
\FloatBarrier

\begin{figure}[!ht]%
    \centering
    {{\includegraphics[width=0.99\textwidth]{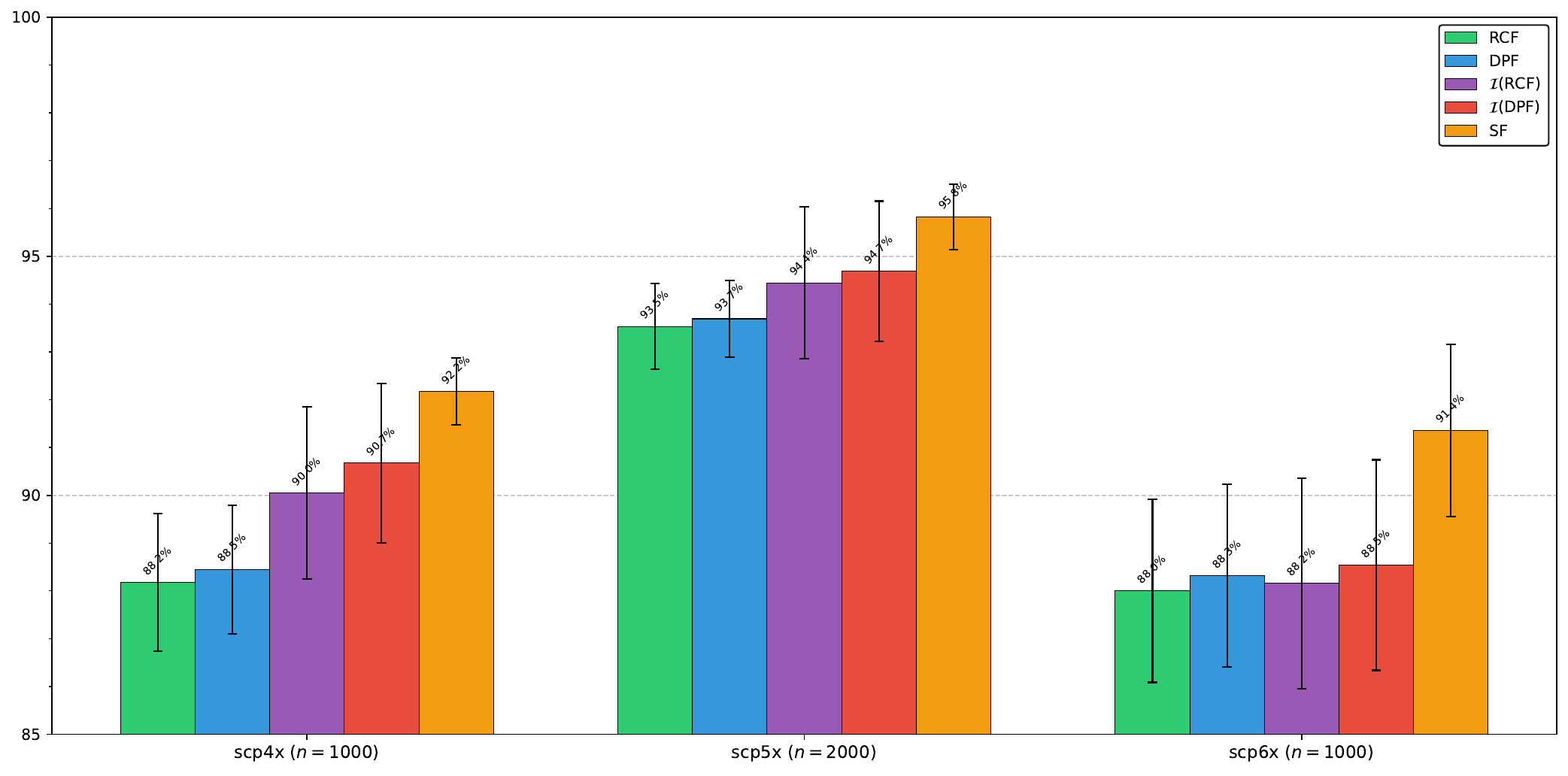} }}%
     \caption{Average percentage  reduction in the number of variables (Beasley instances - sets 4--6)}%
    \label{fig:reduction_Beasleya}%
\end{figure}

\begin{figure}[!ht]%
    \centering
    {{\includegraphics[width=0.99\textwidth]{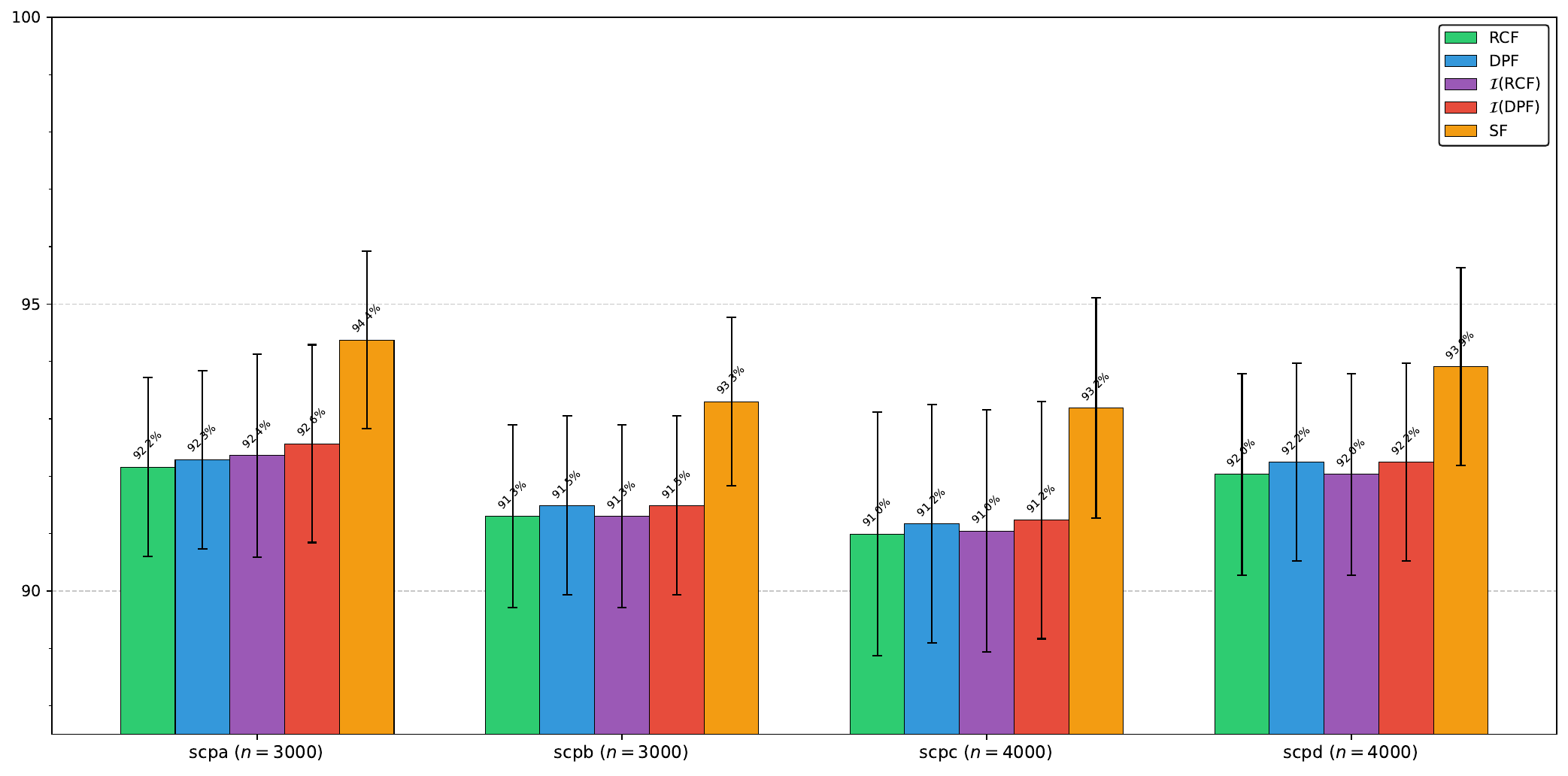} }}%
     \caption{Average percentage  reduction in the number of variables (Beasley instances - sets A--D)}%
    \label{fig:reduction_Beasleyb}%
\end{figure}

\begin{figure}[!ht]%
    \centering
    {{\includegraphics[width=1.01\textwidth]{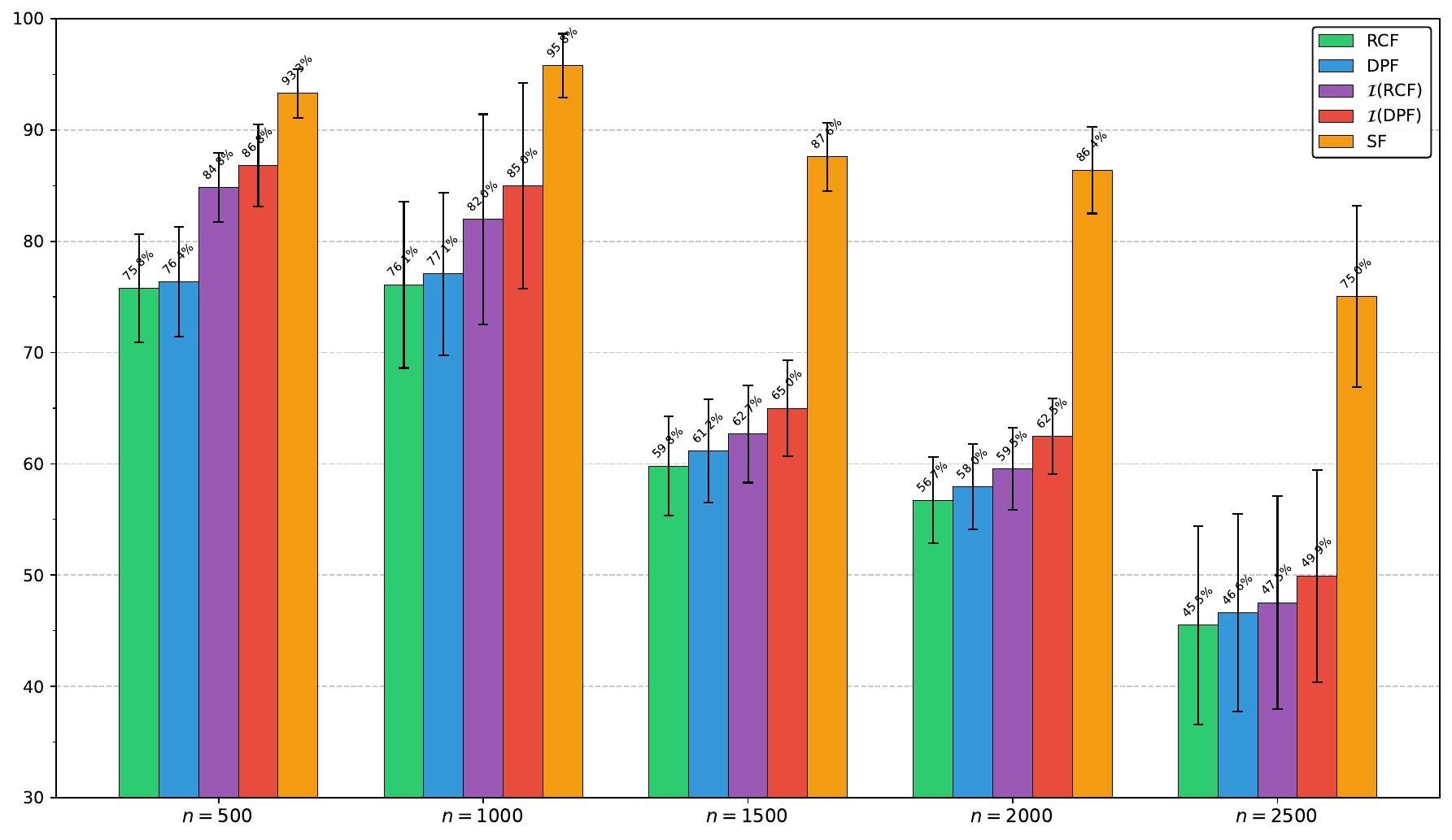} }}%
     \caption{Average percentage reduction in the number of variables (SLS instances)}%
    \label{fig:reduction_SLS}%
\end{figure}

\begin{figure}[!ht]%
    \centering
    {{\includegraphics[width=1.01\textwidth]{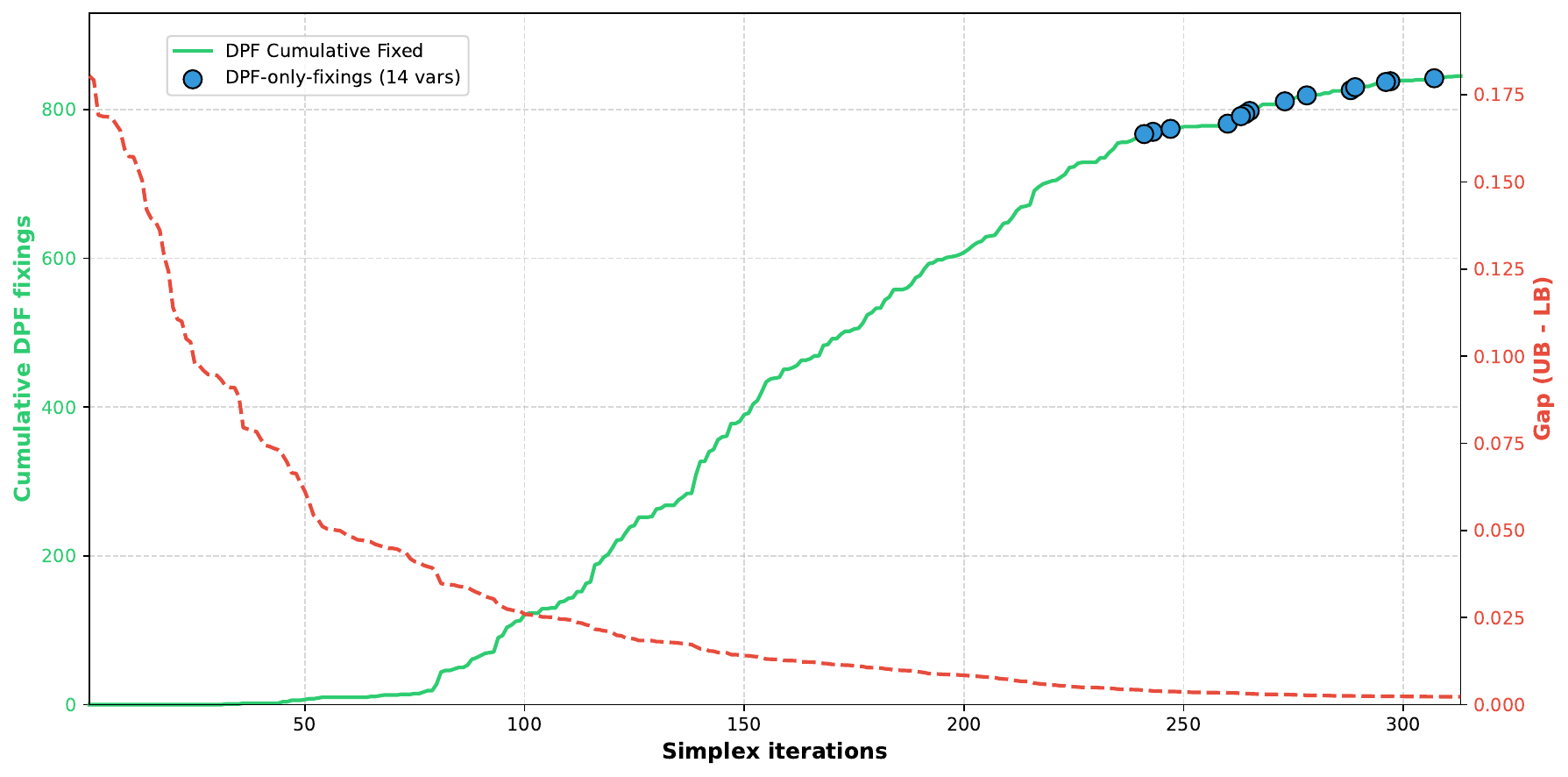} }}%
     \caption{DPF behavior per simplex iteration - SLS instance originally with $n=1000, m=7585$ }%
    \label{fig:simplex_iterations}%
\end{figure}

\begin{figure}[!ht]%
    \centering
    {{\includegraphics[width=1.01\textwidth]{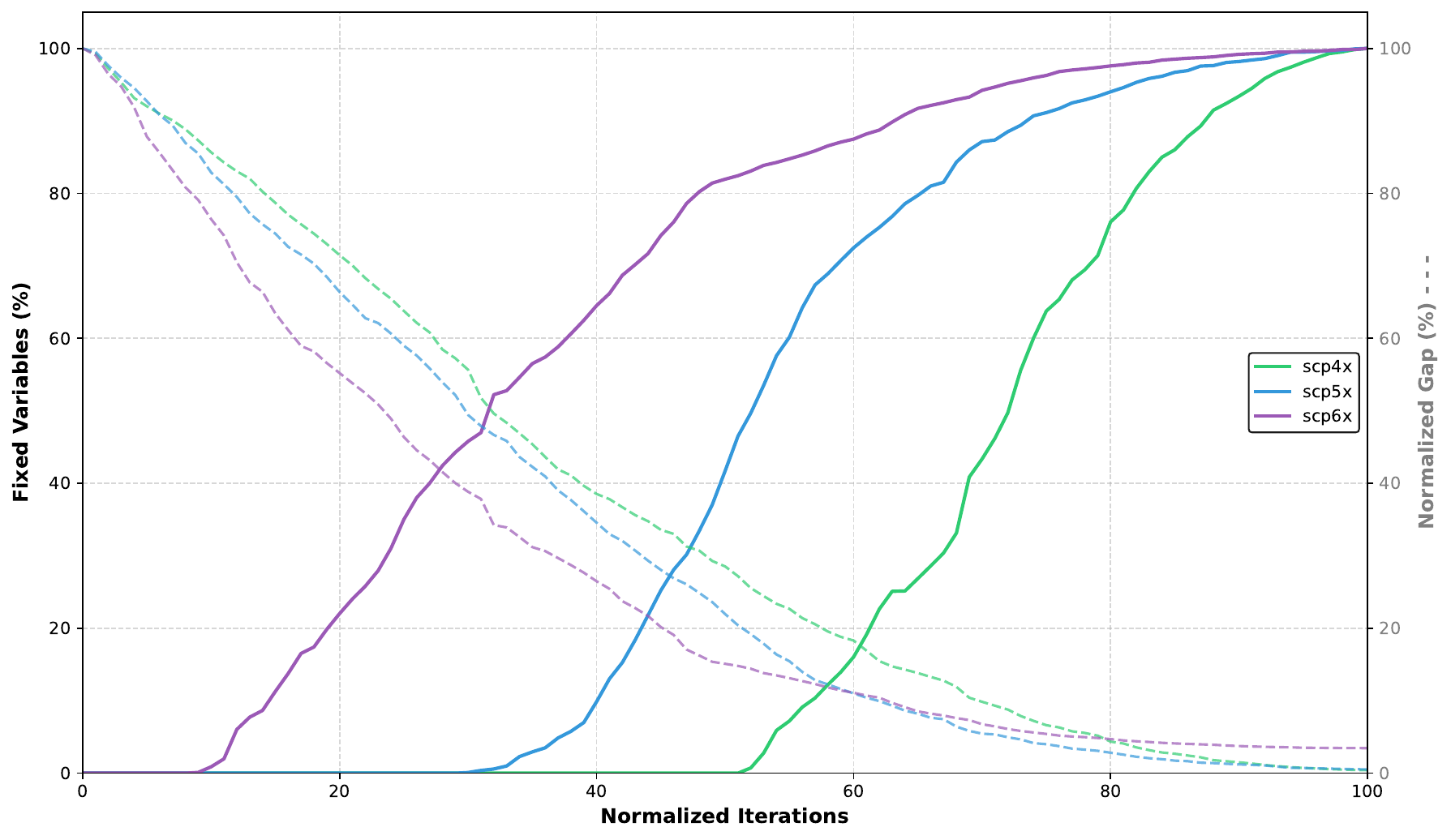} }}%
     \caption{DPF behavior per simplex iteration - Beasley instances - sets 4--6}%
    \label{fig:B_46}%
\end{figure}

\begin{figure}[!ht]%
    \centering
    {{\includegraphics[width=1.01\textwidth]{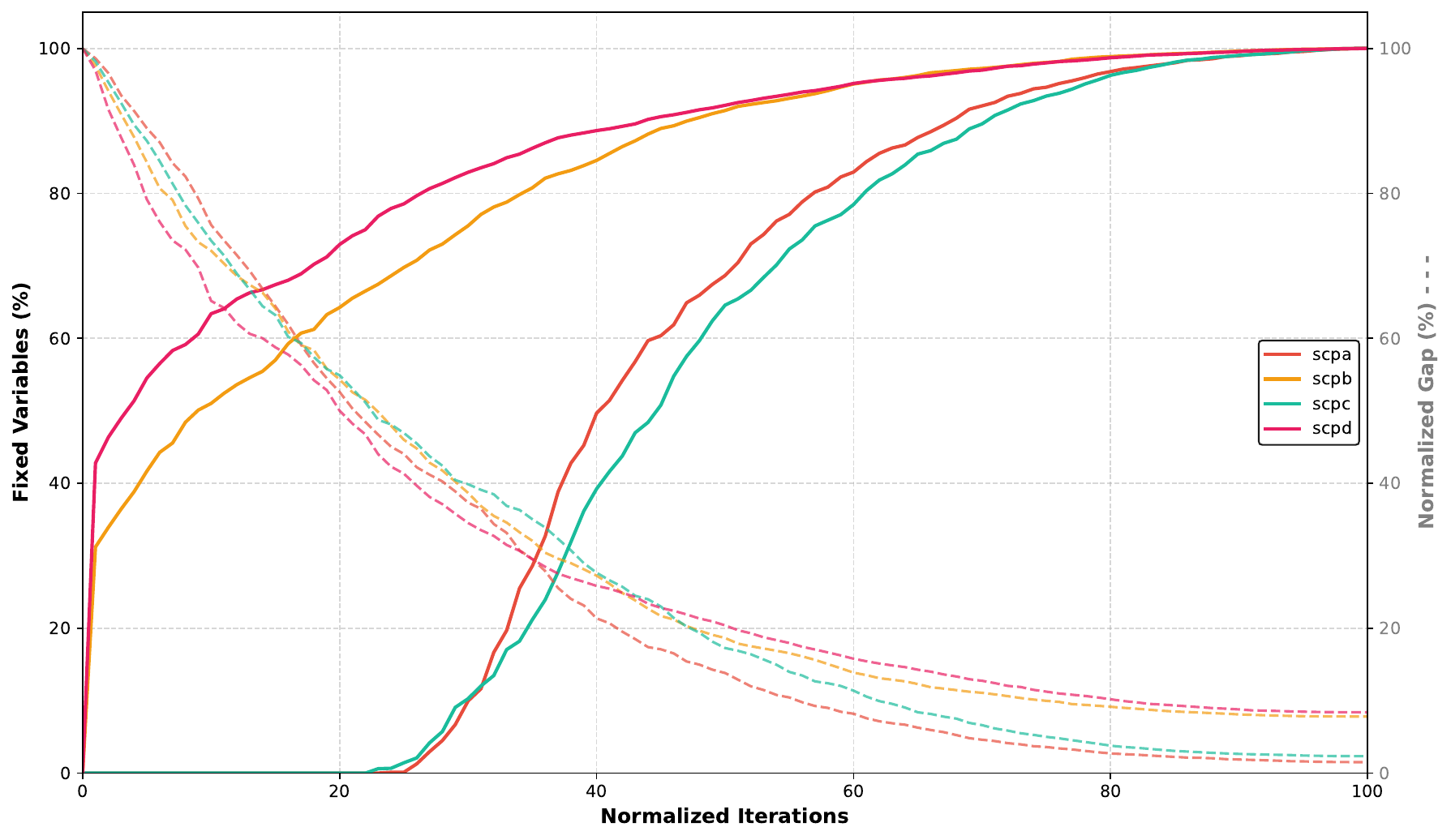} }}%
     \caption{DPF behavior per simplex iteration - Beasley instances - sets A--D}%
    \label{fig:B_AD}%
\end{figure}

\begin{figure}[!ht]%
    \centering
    {{\includegraphics[width=1.01\textwidth]{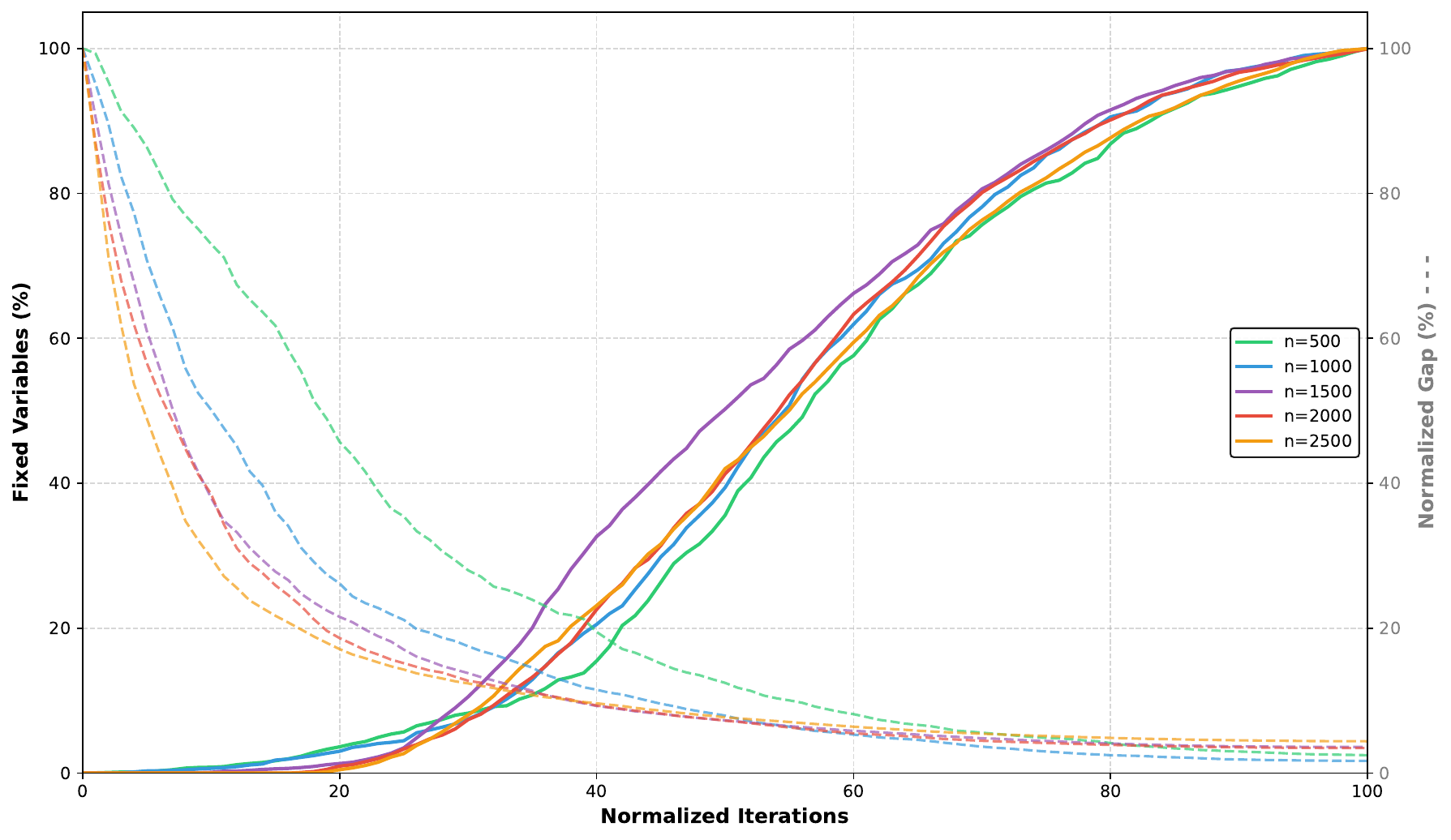} }}%
     \caption{DPF behavior per simplex iteration - SLS instances}%
    \label{fig:SLS}%
\end{figure}

\begin{figure}[!ht]%
    \centering
    {{\includegraphics[width=0.99\textwidth]{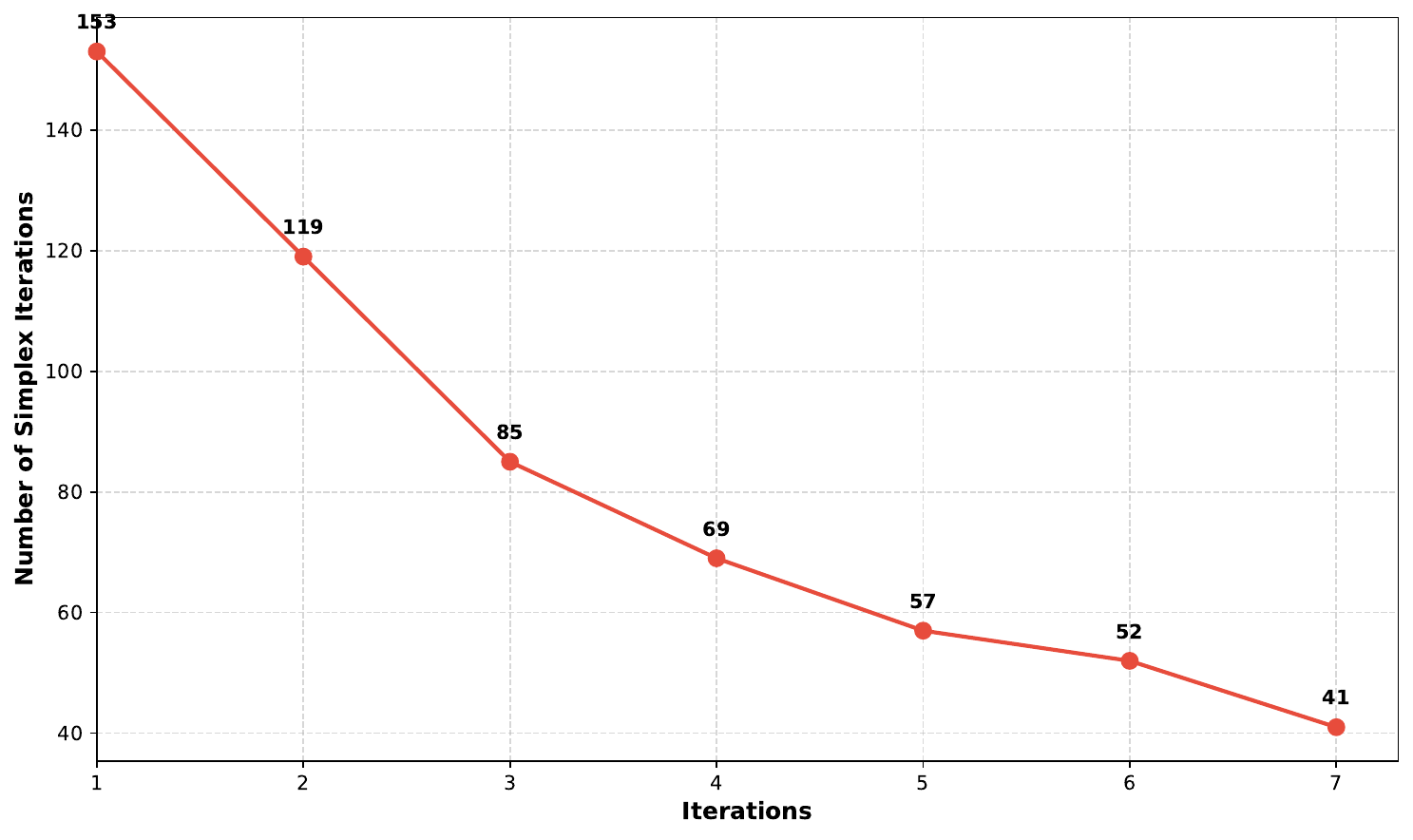} }}%
     \caption{Number of simplex iterations per $\mathcal{I}$(DPF) iteration - SLS instance ($n=1000, m=7585$)}%
    \label{fig:numSimplexit}%
\end{figure}

\begin{figure}[!ht]%
    \centering
    {{\includegraphics[width=1.01\textwidth]{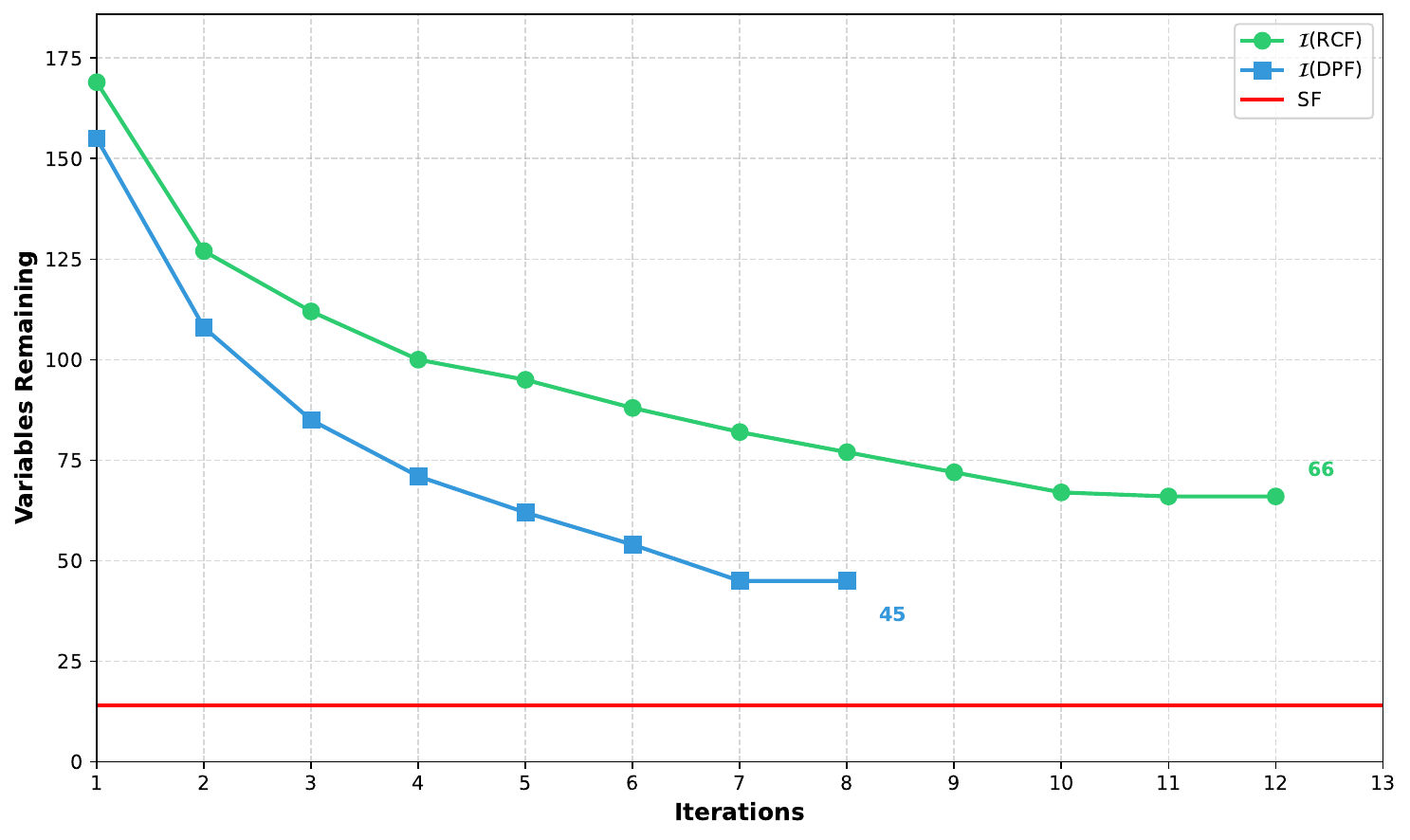} }}%
     \caption{Comparison of iterative procedures - SLS instance originally with $n=1000, m=7585$}%
    \label{fig:iterative_methods}%
\end{figure}

\FloatBarrier
%%%%%%%%%%%%%%%%%%%%%%%%%%%%%%%%%%%

\section{Outlook}\label{sec:outlook}

We have only scratched the surface, computationally,
and we 
have several directions to explore 
to see some practical computational impact of DPF. 
Of course to finally see good impact, 
careful running-time aware implementations will have to be instantiated
(in contrast to the proof-of-concept experiments that we 
have conducted thus far). A difficulty in fully realizing the improvements
that we believe are possible, is that largely MILO and MINLO
software treat the subsolvers as black boxes (although they do allow
warm- and hot-starting). It is interesting to see that such 
issues were successfully overcome in \cite{reducedcostfixing}, 
when they developed, in the early 1980's,
PIPX (which implemented RCF) on top of MPSX-MIP/370: 
``The fact that we were able
 to realize such a system demonstrates anew the feasibility of using
 commercial linear programming package components as building blocks
 in specialized applications''. In modern packages, such as Gurobi, CPLEX and SCIP, user callbacks in the overall solution process give us 
 the ability to intervene to some extent. But it does not appear to be 
 easy to extract, during the overall solution process, iteration-by-iteration solutions from the subsolver. So, to get the best-possible impact from
 DPF, it may be necessary to have a tighter integration of the overall solver
 process and the subsolvers. 

As above, we explain the directions in some detail 
only for problems where the integer variables are binary.
For the basic idea of DPF applied to 0/1 MILO, with relaxations 
solved by the simplex method on the dual, we plan
to investigate:
\begin{itemize}
\item[(i)] Diving experiments aimed at seeing the potential for 
DPF when initialized from an
optimal basis of parent relaxations. A main goal is
to begin to understand the trade-off between starting a
dual algorithm with an advanced initial solution (typical for B\&B)
versus the benefit to DPF of a larger number of iterations. 
\item[(ii)] Using knowledge gained from (i), develop a
 B\&B implementation.
\item[(iii)] Within B\&B, implement a version of 
\emph{strong branching} (see e.g., \cite{StongBranching1} 
and the many references therein),
with a limited number of
simplex-method iterations (on the dual) per strong-branching subproblem
(see \cite[Sec.~2.3]{StongBranching2}, where they mention that
precisely this type of strong branching was already in CPLEX v7.5, released December 2001).
We have seen in Figure \ref{fig:simplex_iterations} that stopping early can give us more than
a good upper bound on the objective value --- 
and applying DPF can give us
an improved lower bound (as compared to 
RCF) on the number of variables that will be fixed
for the integer subproblem. So we plan to use this extra information
to guide branching decisions.
Additionally, we can easily get globally-valid
implication cuts with no additional work
via the ``probing'' (see \cite{Savelsbergh}) 
that comes naturally from
what we have described: (a) for 
a ``down branch'' $x_i=0$: if we can fix $x_j=0$ for
$j\in S$ and $x_j=1$ for $j\in T$, then 
we have the globally-valid implication cuts $x_i \geq x_j$ for
$j\in S$ and $x_i \geq 1-x_j$ for
$j\in T$; (b) for 
an ``up branch'' $x_i=1$: if we can fix $x_j=0$ for
$j\in S$ and $x_j=1$ for $j\in T$, then 
we have the globally-valid implication cuts $x_i \leq 1-x_j$ for
$j\in S$ and $x_i\leq x_j$ for
$j\in T$.
\end{itemize}
Other ideas mentioned and worth investigating involve
applying DPF in the context of: 
\begin{itemize}
\item[(iv)] 0/1 MILO, with dual relaxations solved by
interior point algorithms (see e.g., \cite{Mitchell2,Mitchell1}).
\item[(v)] 0/1 MINLO, 
with subproblem relaxations solved by dual algorithms.
This is most natural for convex 0/1 MINLO
(and is well motivated by the success of variable-fixing based on 
one dual solution for various convex ``MESP'' relaxations; see \cite{FLbook}), 
but as we
saw in \S\ref{sec:LagrangeFixing}, the method could also 
be applied for nonconvex 0/1 MINLO.
\item[(vi)] The use of DPF in combination with early termination of primal-dual algorithms for relaxations of mixed-integer conic formulations that produce a dual-feasible solution at each iteration (e.g., see \cite{ChenGoulart}), with the specific aim of enhancing variable fixing for the mixed-integer second-order-cone formulation in \cite{damcov}. 
\end{itemize}

Finally, we note that, besides strong branching, several well-known MILO algorithms deliberately exploit dual-feasible solutions of the linear-programming (LP) relaxation that are not optimal, often obtained by interrupting the solution process early. For example, in algorithms based on Lagrangian relaxation, subgradient and bundle methods intentionally do not solve the dual problem to optimality. Instead, they are designed to generate a sequence of dual-feasible but non-optimal solutions, each providing a valid dual bound. Similarly, at each iteration of cutting-plane methods, the dual solution of the current LP relaxation is also dual feasible for the original LP relaxation. In all such cases, DPF could also be naturally integrated to enhance the effectiveness of these procedures.

\medskip
\noindent {\bf Acknowledgments.} M. Fampa was supported in part by CNPq grant 307167/2022-4.
J. Lee was supported in part by AFOSR grant FA9550-22-1-0172 and ONR grant
N00014-24-1-2694.

\bibliographystyle{plain}
\bibliography{covering}

\clearpage

\appendix
\section{More numerical results}\label{sec:app}

\begin{table}[htbp]
\centering
\caption{Comparison for Beasley instances (problem dimension $(n, m)$) or 
($n,m\, [\#$it.$]$)}
\label{tab:beasley-preprocessing-comparison}
\small
\setlength{\tabcolsep}{3pt}
\begin{tabular}{l rl rl rl rlc rlc rl}
&\multicolumn{2}{c}{\phantom{A}Original\phantom{AAA}} & \multicolumn{2}{c}{RCF+DRE\phantom{A}} & \multicolumn{2}{c}{DPF+DRE} & \multicolumn{3}{c}{$\mathcal{I}$(RCF+DRE)}& \multicolumn{3}{c}{$\mathcal{I}$(DPF+DRE)} & \multicolumn{2}{c}{SF+DRE} \\
\hline
\hline
scp46 & $1000,$ & $200$ & $126,$ & $129$ & $123,$ & $119$ & $108,$ & $118$ & $[5]$ & $96,$ & $97$ & $[6]$ & $86,$ & $73$ \\
scp48 & $1000,$ & $200$ & $122,$ & $134$ & $118,$ & $125$ & $102,$ & $111$ & $[6]$ & $92,$ & $94$ & $[7]$ & $79,$ & $64$ \\
scp49 & $1000,$ & $200$ & $131,$ & $149$ & $128,$ & $149$ & $118,$ & $134$ & $[5]$ & $116,$ & $133$ & $[5]$ & $81,$ & $74$ \\
scp410 & $1000,$ & $200$ & $94,$ & $93$ & $93,$ & $90$ & $70,$ & $69$ & $[7]$ & $69,$ & $64$ & $[6]$ & $67,$ & $65$ \\
scp51 & $2000,$ & $200$ & $133,$ & $160$ & $128,$ & $147$ & $123,$ & $151$ & $[4]$ & $115,$ & $132$ & $[6]$ & $93,$ & $88$ \\
scp52 & $2000,$ & $200$ & $166,$ & $191$ & $158,$ & $183$ & $166,$ & $191$ & $[3]$ & $157,$ & $183$ & $[3]$ & $88,$ & $76$ \\
scp54 & $2000,$ & $200$ & $124,$ & $123$ & $124,$ & $123$ & $113,$ & $119$ & $[5]$ & $111,$ & $115$ & $[6]$ & $102,$ & $89$ \\
scp56 & $2000,$ & $200$ & $108,$ & $111$ & $105,$ & $105$ & $59,$ & $59$ & $[6]$ & $58,$ & $58$ & $[6]$ & $59,$ & $59$ \\
scp57 & $2000,$ & $200$ & $122,$ & $136$ & $121,$ & $135$ & $108,$ & $115$ & $[6]$ & $104,$ & $109$ & $[6]$ & $83,$ & $79$ \\
scp58 & $2000,$ & $200$ & $123,$ & $156$ & $121,$ & $156$ & $98,$ & $120$ & $[6]$ & $93,$ & $115$ & $[6]$ & $76,$ & $69$ \\
scp61 & $1000,$ & $200$ & $128,$ & $199$ & $127,$ & $199$ & $128,$ & $199$ & $[3]$ & $127,$ & $199$ & $[3]$ & $101,$ & $181$ \\
scp62 & $1000,$ & $200$ & $125,$ & $199$ & $121,$ & $198$ & $125,$ & $199$ & $[3]$ & $121,$ & $198$ & $[3]$ & $94,$ & $180$ \\
scp63 & $1000,$ & $200$ & $117,$ & $188$ & $109,$ & $187$ & $117,$ & $188$ & $[3]$ & $108,$ & $186$ & $[4]$ & $67,$ & $97$ \\
scp64 & $1000,$ & $200$ & $86,$ & $165$ & $85,$ & $164$ & $78,$ & $159$ & $[5]$ & $76,$ & $159$ & $[5]$ & $63,$ & $77$ \\
scp65 & $1000,$ & $200$ & $144,$ & $193$ & $142,$ & $192$ & $144,$ & $193$ & $[3]$ & $141,$ & $192$ & $[3]$ & $107,$ & $152$ \\
scpa1 & $3000,$ & $300$ & $305,$ & $299$ & $300,$ & $299$ & $305,$ & $299$ & $[3]$ & $296,$ & $298$ & $[4]$ & $223,$ & $269$ \\
scpa2 & $3000,$ & $300$ & $252,$ & $285$ & $247,$ & $285$ & $252,$ & $285$ & $[3]$ & $246,$ & $285$ & $[3]$ & $195,$ & $231$ \\
scpa3 & $3000,$ & $300$ & $246,$ & $275$ & $245,$ & $275$ & $241,$ & $275$ & $[3]$ & $240,$ & $275$ & $[3]$ & $198,$ & $238$ \\
scpa4 & $3000,$ & $300$ & $208,$ & $280$ & $204,$ & $279$ & $202,$ & $278$ & $[4]$ & $187,$ & $240$ & $[7]$ & $125,$ & $120$ \\
scpa5 & $3000,$ & $300$ & $165,$ & $215$ & $161,$ & $204$ & $146,$ & $201$ & $[5]$ & $146,$ & $183$ & $[4]$ & $103,$ & $94$ \\
scpb1 & $3000,$ & $300$ & $212,$ & $300$ & $208,$ & $300$ & $212,$ & $300$ & $[2]$ & $208,$ & $300$ & $[2]$ & $150,$ & $278$ \\
scpb2 & $3000,$ & $300$ & $298,$ & $300$ & $291,$ & $300$ & $298,$ & $300$ & $[2]$ & $291,$ & $300$ & $[2]$ & $239,$ & $300$ \\
scpb3 & $3000,$ & $300$ & $258,$ & $300$ & $253,$ & $300$ & $258,$ & $300$ & $[2]$ & $253,$ & $300$ & $[2]$ & $193,$ & $300$ \\
scpb4 & $3000,$ & $300$ & $330,$ & $300$ & $323,$ & $300$ & $330,$ & $300$ & $[2]$ & $323,$ & $300$ & $[2]$ & $263,$ & $300$ \\
scpb5 & $3000,$ & $300$ & $207,$ & $300$ & $202,$ & $300$ & $207,$ & $300$ & $[2]$ & $202,$ & $300$ & $[2]$ & $160,$ & $299$ \\
scpc1 & $4000,$ & $400$ & $273,$ & $363$ & $268,$ & $362$ & $270,$ & $363$ & $[3]$ & $266,$ & $362$ & $[3]$ & $187,$ & $247$ \\
scpc2 & $4000,$ & $400$ & $399,$ & $386$ & $397,$ & $386$ & $398,$ & $386$ & $[3]$ & $396,$ & $386$ & $[3]$ & $320,$ & $373$ \\
scpc3 & $4000,$ & $400$ & $500,$ & $386$ & $487,$ & $386$ & $496,$ & $386$ & $[3]$ & $484,$ & $386$ & $[3]$ & $393,$ & $386$ \\
scpc4 & $4000,$ & $400$ & $356,$ & $365$ & $347,$ & $365$ & $354,$ & $365$ & $[3]$ & $341,$ & $365$ & $[3]$ & $263,$ & $344$ \\
scpc5 & $4000,$ & $400$ & $274,$ & $376$ & $267,$ & $376$ & $273,$ & $376$ & $[3]$ & $266,$ & $376$ & $[3]$ & $199,$ & $280$ \\
scpd1 & $4000,$ & $400$ & $303,$ & $400$ & $291,$ & $400$ & $303,$ & $400$ & $[2]$ & $291,$ & $400$ & $[2]$ & $209,$ & $400$ \\
scpd2 & $4000,$ & $400$ & $369,$ & $400$ & $367,$ & $400$ & $369,$ & $400$ & $[2]$ & $367,$ & $400$ & $[2]$ & $297,$ & $400$ \\
scpd3 & $4000,$ & $400$ & $376,$ & $400$ & $369,$ & $400$ & $376,$ & $400$ & $[2]$ & $369,$ & $400$ & $[2]$ & $313,$ & $400$ \\
scpd4 & $4000,$ & $400$ & $358,$ & $400$ & $340,$ & $400$ & $358,$ & $400$ & $[2]$ & $340,$ & $400$ & $[2]$ & $274,$ & $400$ \\
scpd5 & $4000,$ & $400$ & $188,$ & $400$ & $184,$ & $400$ & $188,$ & $400$ & $[2]$ & $184,$ & $400$ & $[2]$ & $125,$ & $340$ \\
\end{tabular}
\end{table}

\begin{table}[htbp]
\centering
\caption{Comparison for SLS instances: problem dimension ($n,m$) or 
($n,m\, [\#$it.$]$)}
\label{tab:sls-preprocessing-comparison}
\small
\setlength{\tabcolsep}{3pt}
\begin{tabular}{rl rl rl rlc rlc rl}
\multicolumn{2}{c}{\phantom{A}Original\phantom{AAA}} & \multicolumn{2}{c}{RCF+DRE\phantom{A}} & \multicolumn{2}{c}{DPF+DRE} & \multicolumn{3}{c}{$\mathcal{I}$(RCF+DRE)}& \multicolumn{3}{c}{$\mathcal{I}$(DPF+DRE)} & \multicolumn{2}{c}{SF+DRE} \\
\hline
$500,$ & $2468$ & $134,$ & $138$ & $129,$ & $135$ & $66,$ & $61$ & $[7]$ & $59,$ & $55$ & $[7]$ & $32,$ & $20$ \\
$500,$ & $2626$ & $163,$ & $186$ & $162,$ & $186$ & $106,$ & $107$ & $[8]$ & $103,$ & $107$ & $[6]$ & $53,$ & $48$ \\
$500,$ & $2436$ & $98,$ & $98$ & $97,$ & $96$ & $68,$ & $67$ & $[6]$ & $57,$ & $53$ & $[6]$ & $36,$ & $21$ \\
$500,$ & $2451$ & $107,$ & $119$ & $100,$ & $109$ & $65,$ & $71$ & $[5]$ & $57,$ & $60$ & $[5]$ & $23,$ & $21$ \\
$500,$ & $2517$ & $104,$ & $99$ & $103,$ & $98$ & $74,$ & $69$ & $[4]$ & $54,$ & $49$ & $[8]$ & $24,$ & $17$ \\
$1000,$ & $8232$ & $375,$ & $645$ & $362,$ & $612$ & $340,$ & $583$ & $[5]$ & $309,$ & $517$ & $[5]$ & $96,$ & $96$ \\
$1000,$ & $7585$ & $169,$ & $240$ & $155,$ & $226$ & $66,$ & $77$ & $[12]$ & $45,$ & $43$ & $[8]$ & $14,$ & $[14]$ \\
$1000,$ & $7325$ & $175,$ & $246$ & $172,$ & $245$ & $120,$ & $174$ & $[5]$ & $83,$ & $104$ & $[9]$ & $27,$ & $20$ \\
$1000,$ & $8415$ & $225,$ & $321$ & $219,$ & $307$ & $154,$ & $202$ & $[9]$ & $128,$ & $162$ & $[9]$ & $29,$ & $22$ \\
$1000,$ & $8068$ & $252,$ & $398$ & $239,$ & $374$ & $222,$ & $337$ & $[5]$ & $186,$ & $289$ & $[9]$ & $45,$ & $32$ \\
$1500,$ & $15083$ & $659,$ & $1431$ & $646,$ & $1398$ & $628,$ & $1367$ & $[5]$ & $594,$ & $1282$ & $[5]$ & $243,$ & $415$ \\
$1500,$ & $14587$ & $506,$ & $999$ & $482,$ & $931$ & $469,$ & $924$ & $[7]$ & $443,$ & $854$ & $[7]$ & $163,$ & $208$ \\
$1500,$ & $14732$ & $690,$ & $1509$ & $665,$ & $1458$ & $634,$ & $1371$ & $[7]$ & $601,$ & $1298$ & $[8]$ & $236,$ & $379$ \\
$1500,$ & $14474$ & $555,$ & $1132$ & $528,$ & $1062$ & $506,$ & $1006$ & $[7]$ & $466,$ & $917$ & $[9]$ & $164,$ & $220$ \\
$1500,$ & $15440$ & $605,$ & $1221$ & $592,$ & $1173$ & $563,$ & $1124$ & $[5]$ & $522,$ & $1018$ & $[9]$ & $125,$ & $88$ \\
$2000,$ & $24751$ & $924,$ & $2219$ & $906,$ & $2174$ & $866,$ & $2055$ & $[8]$ & $807,$ & $1893$ & $[12]$ & $335,$ & $555$ \\
$2000,$ & $24415$ & $787,$ & $1811$ & $755,$ & $1720$ & $744,$ & $1711$ & $[4]$ & $679,$ & $1519$ & $[14]$ & $202,$ & $280$ \\
$2000,$ & $25769$ & $987,$ & $2544$ & $957,$ & $2470$ & $926,$ & $2362$ & $[7]$ & $855,$ & $2150$ & $[15]$ & $394,$ & $770$ \\
$2000,$ & $24386$ & $830,$ & $1956$ & $805,$ & $1873$ & $764,$ & $1792$ & $[7]$ & $716,$ & $1665$ & $[10]$ & $212,$ & $274$ \\
$2000,$ & $24487$ & $800,$ & $1872$ & $782,$ & $1826$ & $745,$ & $1736$ & $[7]$ & $695,$ & $1612$ & $[7]$ & $219,$ & $336$ \\
$2500,$ & $36217$ & $1012,$ & $2772$ & $989,$ & $2702$ & $936,$ & $2540$ & $[11]$ & $885,$ & $2375$ & $[9]$ & $307,$ & $542$ \\
$2500,$ & $34719$ & $1652,$ & $4671$ & $1634,$ & $4614$ & $1626,$ & $4588$ & $[4]$ & $1566,$ & $4395$ & $[11]$ & $897,$ & $2284$ \\
$2500,$ & $35438$ & $1322,$ & $3705$ & $1281,$ & $3602$ & $1263,$ & $3520$ & $[11]$ & $1202,$ & $3310$ & $[8]$ & $598,$ & $1299$ \\
$2500,$ & $33296$ & $1545,$ & $4192$ & $1508,$ & $4078$ & $1503,$ & $4079$ & $[5]$ & $1449,$ & $3913$ & $[9]$ & $781,$ & $1886$ \\
$2500,$ & $34093$ & $1281,$ & $3470$ & $1259,$ & $3389$ & $1231,$ & $3318$ & $[6]$ & $1156,$ & $3074$ & $[15]$ & $536,$ & $1151$ \\
\end{tabular}
\end{table}

\end{document}